\documentclass{amsart}

\usepackage{enumitem}
\setenumerate{nolistsep}
\setitemize{nolistsep}
\usepackage[square, numbers, comma, sort&compress]{natbib}
\usepackage[final]{pdfpages}
\usepackage{graphicx}
\usepackage{esint}
\usepackage{latexsym}
\usepackage{hyperref}
\usepackage{cleveref}
\usepackage{autonum}

\usepackage{amsmath}
\usepackage{amssymb}

\usepackage{lipsum}

\newcommand\blfootnote[1]{%
  \begingroup
  \renewcommand\thefootnote{}\footnote{#1}%
  \addtocounter{footnote}{-1}%
  \endgroup
}

\usepackage{tikz}
\usetikzlibrary{graphs, angles, quotes,patterns,hobby}
\usepackage{pgfplots}
\pgfplotsset{compat=1.6}

 \DeclareMathOperator{\supp}{supp}

\DeclareMathOperator{\length}{length}
\DeclareMathOperator{\ep}{\epsilon}

\theoremstyle{plain}

\theoremstyle{plain}
\newtheorem*{theor*}{Theorem}

\theoremstyle{plain}
\newtheorem{prop}{Proposition}

\theoremstyle{remark}
\newtheorem{rem}{Remark}
\newtheorem*{rem*}{Remark}

\theoremstyle{definition}

\theoremstyle{plain}

\theoremstyle{remark}

\theoremstyle{plain}
\newtheorem{lemm}{Lemma}

\numberwithin{equation}{section}

\usepackage{hyperref}
\hypersetup{
    colorlinks=true,
    linkcolor=blue,
    filecolor=blue,      
    urlcolor=blue,
    citecolor=red,
}

\begin{document}


\title{Estimates for the maximal Cauchy Integral on chord-arc curves}


\author{Carmelo Puliatti}
\address{BGSMath and Departament de Matem\`atiques, Universitat Aut\`onoma
de Barcelona, 08193, Bellaterra, Barcelona, Catalonia}
\email{puliatti@mat.uab.cat}

\begin{abstract}
We study the chord-arc Jordan curves that satisfy the 
Cotlar-type inequality $T_*(f)\lesssim M^2(Tf),$ where $T$ is the
Cauchy transform, $T_*$ is the maximal Cauchy transform and $M$ is
the Hardy-Littlewood maximal function. Under the background
assumption of asymptotic quasi-conformality we find a
characterization of such curves in terms of the smoothness of a parametrization of the curve.
\end{abstract}

 \maketitle

\section{Introduction}
\blfootnote{\textit{2010 Mathematics Subject Classification.} Primary 42B20, 30C62, 28A80.
\\
\textit{Key words:} Cauchy integral, Cotlar's inequality, asymptotically conformal curve, chord-arc curve.}
Consider a homogeneous smooth Calder\'{o}n-Zygmund operator in $\mathbb{R}^n$
\begin{equation}\label{CZ}
Tf(x)= \operatorname{p.v.} \int f(x-y)\,K(y)\,dy \equiv \lim_{\epsilon\rightarrow
0} T_\epsilon f(x), \quad x \in \mathbb{R}^n,
\end{equation}
where $T_\epsilon$ is the truncation at level $\epsilon$ defined by
\begin{equation}
T_\epsilon f(x)= \int_{| y| > \epsilon} f(x-y) K(y) \,dy, \quad x \in \mathbb{R}^n,
\end{equation}
and $f$ is in $L^p(\mathbb{R}^n), \; 1 \le p < \infty.$ Here the kernel $K$ is of class $C^{\infty}$ off the origin,
homogeneous of order $-n$ and with zero integral on the unit sphere $$\{x \in \mathbb{R}^n : |x|=1\}.$$
Let $T_{*}$ be the maximal singular integral
\begin{equation}
T_{*}f(x)= \sup_{\epsilon > 0} | T_\epsilon f(x)|, \quad x \in \mathbb{R}^n,
\end{equation}

A classical fact relating $T_{*}$ and the standard Hardy-Littlewood maximal operator $M$ is Cotlar's inequality, which reads
\begin{equation}\label{Cotlar}
T_{*}(f)(x)\leq C\,\big(M(Tf)(x) + M(f)(x)\big), \quad x \in
\mathbb{R}^n.
\end{equation}
Combining this with the $L^p$ estimates  $\|T(f)\|_p \le  C \,\|f\|_p $ and  $\|M(f)\|_p \le  C \,\|f\|_p $, \; $1<p< \infty$ one gets  $\|T_{*}(f)\|_p \le  C \,\|f\|_p, \; 1<p< \infty$.

It was discovered in \cite{MOV} that if $T$ is an even  higher order Riesz transform, that is, if
$K(x) = P(x)/|x|^{n+d}$, with $P$ an even homogeneous polynomial of degree $d$, then one can get rid of the
second term in the right hand side of \eqref{Cotlar}, namely,
\begin{equation}\label{SuperCotlar}
T_{*}(f)(x)\leq C\, M(Tf)(x), \quad x \in \mathbb{R }^n.
\end{equation}
Hence  $\|T_{*}(f)\|_p \le  C \,\|T(f)\|_p, \; 1<p< \infty$, in this case. However,  if $T$ is an odd higher order Riesz transform, then
\eqref{SuperCotlar} may fail and the right substitute turns out to be (see \cite{MOPV})
\begin{equation}\label{CotlarOdd}
T_{*}(f)(x)\leq C\, M^2(Tf)(x), \quad x \in \mathbb{R }^n,
\end{equation}
where $M^2$ stands for the iteration of $M$.

Inequalities of the type \eqref{SuperCotlar} and \eqref{CotlarOdd}
were first considered in relation to the David-Semmes problem (see
\cite{MOV},\cite{MOPV} and \cite{V}) and later on were studied in
the context of the Cauchy singular integral on Lipschitz graphs and
$C^1$ curves by Girela-Sarr\'{\i}\'{o}n in \cite{G}. Let $\Gamma$ be either
a Lipschitz graph or a closed chord-arc curve in the plane, let $T$
be the Cauchy Singular Integral and $M$ the Hardy-Littlewood maximal
operator, both with respect to the arc-length measure, and let
$T_{*}$ be the maximal Cauchy Integral. Precise definitions will be
given below. Girela-Sarri\'on showed in \cite{G} that the presence
at a point $z$ of the curve of a non-zero angle prevents
\eqref{CotlarOdd}, with $x$ replaced by $z$, to hold. This agrees
with the intuition that \eqref{CotlarOdd}  should help in finding
tangent lines, but suggests that it is a condition definetely
stronger than the mere existence of tangents. It was also shown in
\cite{G} that if $\Gamma$ is a closed $C^1$ curve with the property
that the modulus of continuity $\omega(z, \delta) $ of the unit
tangent vector satisfies
\begin{equation}\label{logomega}
\omega(z,\delta) \le C\, \frac{1}{\log(\frac{1}{\delta})},  \quad z \in \Gamma, \quad  \delta <  1/2,
\end{equation}
then \eqref{CotlarOdd} holds with $x \in \mathbb{R}^n$ replaced by $z \in \Gamma$.

 In this paper we study the validity of inequality \eqref{CotlarOdd} in the context of chord-arc curves. A chord-arc curve is a
rectifiable Jordan curve $\Gamma$ in the plane with the property that there exists a positive constant $C$ such that, given any two
points $z_1, z_2 \in \Gamma$ one has
\begin{equation}\label{chordarc}
\l(z_1,z_2) \leq C \,|z_1-z_2 |,
\end{equation}
where $l(z_1,z_2)$  is the length of the shortest arc in $\Gamma$ joining $z_1$ and $z_2$.  Equivalently $\Gamma$ is a
bilipschitz image of the unit circle (see \cite{Pomm_bd}, Theorem 7.9).
 Then $\Gamma$ can be parametrized by a
periodic function $ \gamma : \mathbb{R} \rightarrow \Gamma $ of period $T$ satisfying the bilipschitz condition
\begin{equation}\label{bilipschitz}
 \frac{1}{L}\,|x-y| \leq |\gamma(x) -\gamma(y)| \leq L \,|x-y|, \quad x, y \in \mathbb{R}, \quad |x-y| \leq \frac{T}{2},
\end{equation}
for some positive constant $L$. We say, by slightly abusing language, that $\gamma$ is a bilipschitz parametrization of $\Gamma$.
One can take, for instance,  the $T$-periodic extension of the arc-length parametrization of $\Gamma$ with $T$ being the length of $\Gamma$.

 One can easily define the maximal Hardy-Littlewood operator and the Cauchy Integral on a chord-arc curve. Given $z\in \Gamma$ let
 $t \in \mathbb{R}$ be such that $z=\gamma(t).$      Set
\begin{equation}
\Gamma_{z,r}:=\gamma(\{\tau :|\tau-t|<r\}).
\end{equation}
One should look at $\Gamma_{z,r}$ as ``balls'' of
radius $r$ centered at $z$. Indeed, owing to the bilipschitz condition \eqref{bilipschitz}, each $\Gamma_{z,r}$  contains and is contained in a disc in
$\Gamma$ of radius comparable to $r$, for $r < T .$  It will be more convenient
to work with $\Gamma_{z,r}$  than with the euclidean discs $D(z,r)\cap \Gamma$, where $D(z,r)$
 stands for the planar disc of center $z$ and radius $r$.

 Denote by $\mu$ the arc-length measure on $\Gamma$. For $f\in L^1(\Gamma,\mu)$ and $z\in\Gamma,$ we define the Hardy-Littlewood maximal
 function on the curve $\Gamma$ as
\begin{equation}
Mf(z):=\underset{r>0}{\sup}\,\frac{1}{\mu(\Gamma_{z,r})}\int_{\Gamma_{z,r}}|f|d\mu.
\end{equation}
The Cauchy Integral is defined as
 \begin{equation}\label{Cauchy}
Tf(z)= \operatorname{p.v.} \frac{1}{\pi\, i}\int_\Gamma \frac{1}{w-z}\,f(w)\,dw \equiv \lim_{\epsilon\rightarrow
0} T_\epsilon f(z), \quad z \in \Gamma,
\end{equation}
 where
 \begin{equation}\label{Cauchytrunc}
T_{\epsilon}f(z)=\frac{1}{\pi i}\int_{\Gamma\backslash\Gamma_{z,\epsilon}}\frac{f(w)}{w-z}dw
\end{equation}
is the truncated Cauchy Integral at level $\epsilon$. The maximal Cauchy Integral is
\begin{equation}\label{Cauchymax}
T_*f(z):=\underset{\epsilon>0}{\sup}\big|T_{\epsilon}f(z)\big|.
\end{equation}

Our aim is to investigate under what conditions on $\Gamma$ one has the inequality
\begin{equation}
 T_* f(z)\leq C\, M^2(Tf)(z), \quad z \in \Gamma, \quad f \in L^2(\Gamma, \mu),
\end{equation}
where $C$ is a positive constant. Since we know that angles prevent
the above inequality to hold, we need to require on $\Gamma$ a
condition that excludes them. One such a condition is asymptotic
quasiconformality. Given two points $z_1, z_2 \in \Gamma$ let
$A(z_1,z_2)$ be the arc in $\Gamma$ joining the two points and
having smallest diameter (there is only one if the two points are
sufficiently close). The Jordan curve $\Gamma$ is said to be
asymptotically conformal if, given a positive number $\delta$ there
exists a positive $\epsilon$, so that for any two points $z_1, z_2
\in \Gamma$ satisfying $|z_1-z_2| < \epsilon$ one has
\begin{equation}\label{ac}
 |z_1 -z |+|z_2 -z| \leq (1+\delta)|z_1 -z_2|, \quad z \in A(z_1,z_2).
\end{equation}

Our main result reads as follows.
\begin{theor*}\label{teorem} Let $T$ be the Cauchy Integral on an asymptotically conformal  chord-arc curve $\Gamma$ and let $\gamma$ be a bilipschitz parametrization of  $\Gamma$.
 Then the estimate
\begin{equation}\label{m2}
T_*(f)(z) \leq  C\, M^2(Tf)(z), \quad z\in \Gamma, \quad f \in
L^2(\Gamma,\mu),
\end{equation}
holds if and only if there exists $C>0$ such that
\begin{equation}\label{snddiff}
\big| \gamma(x+\epsilon)+\gamma(x-\epsilon)-2\gamma(x)\big|\leq C \,
\frac{\epsilon}{|\log\epsilon|},
\end{equation}
for each $\ep$ satisfying $0 < \epsilon< T$ and for each $x\in\mathbb{R}.$
\end{theor*}

One should recall that condition \eqref{snddiff} implies that
$\gamma$ is differentiable almost everywhere in the ordinary sense
and the derivative is a function of vanishing mean oscillation
(see \cite{WeissZygm}). Therefore, for chord arc curves satisfying the background
assumption of asymptotical conformality, inequality \eqref{m2} is
equivalent to the precise form of differentiability described in
terms of second order differences in \eqref{snddiff}.

In Section 2 we prove a couple of Lemmas which allow to express
condition \eqref{m2} in an equivalent form in terms of a function
related to the geometry of $\Gamma.$ Section 3 is devoted to take
care of a technical question, namely, that it is enough to estimate
truncations at small enough levels. In Section 4 we prove the
Theorem by means of three lemmas, one on them making the connection
between the function carrying the geometrical information and the
second difference condition \eqref{snddiff}. In Section 5 we present
an example of a spiralling domain that enjoys the equivalent
conditions in the Theorem but whose boundary is not of class $C^1.$

Our terminology and notation are standard. We let $C$ denote a
constant independent of the relevant variables under consideration
and which may vary at each occurrence. The notation $A \lesssim B$
means that there exists a constant $C>0$ such that $A\leq C B.$ We write $A\gtrsim B$ if $B\lesssim A.$ The disc centered at $z$ of radius $r$ is denoted by
$D(z,r)$.

\section{Two preliminary Lemmas}
The beginning of the proof follows the ideas of \cite{G}, so that we
will be rather concise. Given a function $f\in L^1(\Gamma,\mu)$ we
denote by
$m_{\Gamma_{z,\epsilon}}(f)=\fint_{\Gamma_{z,\epsilon}}f(w)\,d\mu(w)$
the mean of $f$ on $\Gamma_{z,\epsilon}$ with respect to the arc length measure $\mu$.  We let $K_{z,\epsilon}$
denote the Cauchy kernel truncated at the point $z$ at level
$\epsilon,$ that is,
\begin{equation}\label{Cauchytrun}
K_{z,\epsilon}(w)= \frac{1}{\pi \,i} \frac{1}{w-z}
\,\chi_{\Gamma\setminus \Gamma_{z,\epsilon}}(w),\quad w \in \Gamma.
\end{equation}
Set $g_{z,\epsilon}=T(K_{z,\epsilon})$ and let $M>1$ be a big number to
be chosen later. Following \cite[p.673]{G} we obtain the identity
\begin{equation}\label{identitat}
-T_{\epsilon}f(z)=I_{\epsilon} + II_{\epsilon} + III_{\epsilon},
\end{equation}
where
\begin{align}\label{123_1}
I_{\epsilon}&:=
\int_{\Gamma_{z,M\epsilon}}Tf(w)\left(g_{z,\epsilon}(w)-m_{\Gamma_{z,M\epsilon}}(g_{z, \epsilon})\right)dw,\\
\label{123_2} II_{\epsilon} &:=
m_{\Gamma_{z,M\epsilon}}(g_{z,\epsilon})\int_{\Gamma_{z,M\epsilon}}
Tf(w)dw \\    \label{123_3} III_{\epsilon} &:= \int_{\Gamma\setminus
\Gamma_{z,M\epsilon}}Tf(w)g_{z,\epsilon}(w)dw.
\end{align}
Following closely the argument in  \cite{G} one can prove
\begin{align}\label{stima1}
|I_{\epsilon}|\leq C\, M^2(Tf)(z),\\ \label{stima2}
|II_{\epsilon}|\leq C\, M(Tf)(z).
\end{align}
Since clearly $M(g) \leq M^2(g)$ for any $g$, we are left with the
task of estimating $III_\epsilon.$  The next lemma provides an
expression for $III_\epsilon$ in terms of a function encoding the
smoothness of $\Gamma.$  To state the lemma first we need to clarify
the definition of a branch of the logarithm of $w-z$, as a
function of $w$ with $z \in \Gamma$ fixed, in an appropriate region.

Given $z\in \Gamma$ let $\Delta_z$ be a curve connecting $z$ and
$\infty$ in the unbounded component of $\mathbb{C}\setminus \Gamma.$
Such curves exist and indeed we will construct a special one in
Section 4 (under the additional assumption of asymptotic
quasiconformality). Hence $\mathbb{C}\setminus \Delta_z$ is a simply
connected domain containing $\Gamma \setminus \{z\}$ and so there
exists in $\mathbb{C}\setminus \Delta_z$ a branch of $\log(w-z).$ In
particular, if $z=\gamma(x)$ for some $x \in \mathbb{R}$, the
expressions $\log (\gamma(x+\epsilon)-\gamma(x))$ and $\log
(\gamma(x-\epsilon)-\gamma(x))$ make sense for $ 0 < \epsilon < T. $

\begin{lemm}\label{lemm_curvature}
Let $\Gamma$ be a chord-arc curve and $\gamma$ a bilipschitz
parametrization of $\Gamma$. Let $z \in \Gamma$ and let $x$ be
 a real number such that $\gamma(x)=z$. Then
    for almost every $w\in\Gamma\backslash\Gamma_{z,M\epsilon}$ we have
    \begin{equation}
    T(K_{z,\epsilon})(w)=\frac{1}{\pi^2 (z-w)}\big[ F(x,\epsilon)+G_{z,\epsilon}(w)\big],
    \end{equation}
    where
    \begin{equation}
    \label{BBB}
    F(x,\epsilon)= \log(\gamma(x+\epsilon)-\gamma(x)) - \log(\gamma(x-\epsilon)-\gamma(x)) + \pi i
    \end{equation}
    and
    \begin{equation}\label{decG}
    | G_{z,\epsilon}(w)|\leq
    \frac{C\,\epsilon}{|z-w|}.
    \end{equation}
\end{lemm}

\begin{proof}
Take $w\in\Gamma\backslash\Gamma_{z,M\epsilon}$ . Then

\begin{align}
T(K_{z,\epsilon})(w) & = -\frac{1}{\pi^2} \underset{\delta \rightarrow 0}{\lim}\int_{\Gamma\setminus(\Gamma_{w,\delta }\cup \Gamma_{z,\epsilon})}\frac{1}{(\zeta-z)(\zeta-w)}\,d\zeta \\
& = -\frac{1}{\pi^2}  \frac{1}{w-z} \,\underset{\delta \rightarrow 0}{\lim} \int_{\Gamma\setminus(\Gamma_{w,\delta }\cup \Gamma_{z,\epsilon}) }\left( \frac{1}{\zeta-w} - \frac{1}{\zeta-z} \right)\,d\zeta.
\end{align}
Let $y \in \mathbb{R} $ with $\gamma(y) = w$ . Then the latest integral in the above formula is
\begin{align}
& \log\left(\gamma(y-\delta) - \gamma(y) \right) - \log\left(\gamma(x+\epsilon) - \gamma(y) \right)
 +\log\left(\gamma(x-\epsilon) - \gamma(y) \right) \\
 & - \log\left(\gamma(y+\delta) - \gamma(y) \right)
 -\Big( \log\left(\gamma(y-\delta) - \gamma(x) \right) \\
&- \log\left(\gamma(x+\epsilon) - \gamma(x)\right)
+\log\left(\gamma(x-\epsilon) - \gamma(x)\right)-  \log\left(\gamma(y+\delta) - \gamma(x) \right)     \Big).
\end{align}
Assume that $\gamma$ is differentiable at the point $y$ and the derivative $\gamma'(y)$
does not vanish. Then we have that
\begin{equation}
\lim_{\delta \rightarrow 0} \Big(\log\left(\gamma(y-\delta) - \gamma(y) \right)- \log\left(\gamma(y+\delta) - \gamma(y) \right)\Big) = \pi i,
 \end{equation}
 because the curve $\Delta_w$ lies in the unbounded component of $\mathbb{C} \setminus \Gamma$, and then to the right hand side of $\Gamma$,
 oriented according to the parametrization $\gamma$.
 Taking limit as $\delta$ goes to $0$ we obtain
 \begin{align}
T(K_{z,\epsilon})(w) & = -\frac{1}{\pi^2}  \frac{1}{w-z} \Big( \big(\log(\gamma(x+\epsilon) - \gamma(x)) - \log(\gamma(x-\epsilon) - \gamma(x))+ \pi i \big) \\
& + \big(\log(\gamma(x+\epsilon) - \gamma(y)) - \log(\gamma(x-\epsilon) - \gamma(y))\big)\Big).
\end{align}
Define
\begin{equation}\label{exprG}
G_{z,\epsilon}(w)= \log\left(\gamma(x+\epsilon) - \gamma(y)\right) - \log\left(\gamma(x-\epsilon) - \gamma(y)\right).
\end{equation}
\noindent It remains to show the decay inequality \eqref{decG}. According to the choice of $\Delta_w$ we have a well defined branch of
$\log(\gamma(x+t)-w), \;  -\epsilon < t < \epsilon.$  Thus
\begin{equation}\label{intG}
G_{z,\epsilon}(w) = \int_{-\epsilon}^{\epsilon} \frac{d}{dt} \log(\gamma(x+t)-w) \,dt =
 \int_{-\epsilon}^{\epsilon} \frac{\gamma'(x+t)}{\gamma(x+t)-w}\, dt.
\end{equation}
Since $w=\gamma(y) \in \Gamma \setminus \Gamma_{z, M \epsilon}$, we have $ y \notin (x-M \epsilon, x+M \epsilon)$ and so
\begin{equation}\label{sota}
|w-z| = |\gamma(y) -\gamma(x)| \geq  \frac{|y-x|}{L} \geq \frac{M \epsilon}{L} ,
\end{equation}
which gives , taking $M \geq 2 L^2$,
\begin{align}
|w-\gamma(x+t)|  & \geq  |w-z| - |\gamma(x) -\gamma(x+t)| \\ & \geq \frac{ |w-z| }{2}  + \frac{M \epsilon}{2 L} - L \epsilon \\ & \geq \frac{ |w-z| }{2}.
\end{align}
Hence, by \eqref{intG},
\begin{align}
|G_{z,\epsilon}(w) | & \leq  \int_{-\epsilon}^{\epsilon} \frac{|\gamma'(x+t)|}{|\gamma(x+t)-w|}\, dt \leq \frac{4L \epsilon}{|w-z|}.\qedhere
\end{align}
\end{proof}

\begin{lemm}\label{Flog}
Let $\Gamma$ be a chord-arc curve and $\gamma$ a bilipschitz
parametrization of $\Gamma$.  Then the inequality
\begin{equation}\label{m22}
T_*(f)(z) \leq  C\, M^2(Tf)(z), \quad z\in \Gamma, \quad f \in
L^2(\Gamma,\mu),
\end{equation}
is equivalent to
\begin{equation}\label{Ffitada}
| F(x,\epsilon)| |\log(\epsilon)| \leq C\, \quad 0< \epsilon <T, \quad x \in \mathbb{R}.
\end{equation}
\end{lemm}

\begin{proof}
Assume that \eqref{Ffitada} holds. Then by Lemma \ref{lemm_curvature}
\begin{align}
 III_{\epsilon}  & = \int_{\Gamma \setminus \Gamma_{z,M \epsilon}} Tf(w) \, T(K_{z,\epsilon})(w) \,dw \\
& = \frac{F(x,\epsilon)}{\pi^2} \int_{\Gamma \setminus \Gamma_{z,M \epsilon}}  \frac{Tf(w)}{z-w}\,dw + \frac{1}{\pi^2} \int_{\Gamma \setminus \Gamma_{z,M \epsilon}} Tf(w) \, \frac{G_{z,\epsilon}(w) }{z-w}\,dw \\
& =  F(x,\epsilon) \, IV_\epsilon + V_\epsilon,
\end{align}
where the last identity is a definition of the terms $IV_\epsilon $
and  $V_\epsilon$. One can break the domain of integration in the integrals in $IV_\epsilon $
and  $V_\epsilon$ into a union of dyadic annuli
\begin{equation}
A_j = \gamma \big\{y \in \mathbb{R} :   M \epsilon\, 2^j < | y-x| \leq  M \epsilon \, 2^{j+1} \big\}, \quad j=0, 1, ...
\end{equation}
then perform standard estimates and apply \eqref{decG} to get, thanks to the quadratic decay of the integrand,
\begin{equation}\label{quatre}
|  V_\epsilon |  \leq C\, M\big(T(f)\big)(z).
\end{equation}
For $IV_\epsilon $ one only has a first order decay, which gives
\begin{equation}
|  IV_\epsilon | \leq C\,  \Big|\log\Big(\frac{ML}{\epsilon}\Big)\Big|M(Tf)(z),
\end{equation}
thus completing the proof of the sufficient condition.

Assume now  \eqref{m22}.  Recalling that $III_{\epsilon} =F(x,\epsilon) \, IV_\epsilon + V_\epsilon$ and \eqref{quatre},
we obtain
\begin{equation}\label{F4}
\big|  F(x,\epsilon) \, IV_\epsilon \big|  \leq C\, M^2\big(T(f)\big)(z), \quad z \in \Gamma, \quad f \in L^2(\Gamma, \mu).
\end{equation}
The Cauchy Singular Integral operator $T$ is an isomorphism of $L^2(\Gamma, \mu)$ onto itself. This is proved in
Lemma 1 of \cite[p. 661]{G}  for Lipschitz graphs, and the same proof works in our context. Thus \eqref{F4} can be rewritten as
\begin{equation}\label{Fg}
\Big|  F(x,\epsilon) \,  \int_{\Gamma \setminus \Gamma_{z,M \epsilon}}  \frac{g(w)}{z-w}\,dw  \Big|  \leq C\, M^2(g)(z), \quad z \in \Gamma, \quad g \in L^2(\Gamma,\mu).
\end{equation}
To simplify the notation take $x=0= \gamma(x).$ Assume first that $0 <
\epsilon <  1.$  Apply \eqref{Fg} with $g$ the characteristic
function of $\gamma((\epsilon^n, \epsilon))$, where $n$ is a large
integer to be chosen. Then
\begin{equation}\label{F0}
|  F(0,\epsilon)|  \Big| \int_{\ep^n}^{\ep} \frac{\gamma'(t)}{\gamma(t)}
\,dt \Big| \leq C
\end{equation}
and
\begin{align}
 \Big| \int_{\ep^n}^{\ep} \frac{\gamma'(t)}{\gamma(t)} \,dt \Big| & =  | \log(\gamma(\ep)) - \log(\gamma(\ep^n) )| \\
 & \geq  | \log(|\gamma(\ep)|) - \log(|\gamma(\ep^n) |)|\\
 & \geq \log\Big(\frac{1}{L^2\,\ep^{n-1}}\Big)\\
 & \geq -2 \log(L) + (n-2) \log\Big(\frac{1}{\ep}\Big) +
 \log\Big(\frac{1}{\ep}\Big)\\
 & \geq |\log(\ep)|
\end{align}
provided $n=n(\ep)$ is large enough so that  $-2 \log(L) + (n-2)
\log({1}/\ep) \geq 0$. Therefore  \eqref{Ffitada} follows in
this case.

If $1 \leq \ep < T$ then we take as $g$ the characteristic function
of $\gamma((\ep^{-n}, \ep)).$ In this case we get
\begin{align}
 \Big| \int_{\ep^{-n}}^{\ep} \frac{\gamma'(t)}{\gamma(t)} \,dt \Big| & \geq  -2 \log(L) +
 n \log(\ep)+ \log(\ep)\\ &
 \geq |\log(\ep)|
\end{align}
provided $n$ is chosen so that $-2 \log(L) + n \log(\ep) \geq 0$.
\end{proof}

\section{Reduction to estimating truncations at small levels}
In this section we reduce the proof of \eqref{m2} to estimating the
truncations $T_{\ep} f$ for small $\ep.$ In the previous section we
showed that the estimate of $T_{\ep} f$ can be reduced to that of
the term $III_{\ep}$ in \eqref{123_3}.

\begin{lemm}\label{smalltrunc} If $\epsilon_0$ is a given positive number, then there exists a large positive number $M=M(L)$
so that
\begin{equation}
\Big|\int_{\Gamma\setminus \Gamma_{z,M
\epsilon}}Tf(w)\,g_{z,\epsilon}(w)dw \Big| \leq C\, M(Tf)(z),\quad z
\in \Gamma, \quad \epsilon_0 <\epsilon,
\end{equation}
for a positive constant $C=C(\ep_0,L).$
\end{lemm}

The small number $\ep_0$ will be chosen in the next section.

\begin{proof} Recall that
\begin{align}
g_{z,\ep}(w) &= T(K_{z,\epsilon})(w) \\& =
-\frac{1}{\pi^2} \operatorname{p.v.} \int_{\Gamma\setminus \Gamma_{z,\epsilon}} \frac{1}{(\zeta-w)(\zeta-z)}\,d\zeta \\
& = -\frac{1}{\pi^2}  \frac{1}{w-z} \, \operatorname{p.v.}
\int_{\Gamma\setminus \Gamma_{z,\epsilon}} \left( \frac{1}{\zeta-w}
- \frac{1}{\zeta-z} \right)\,d\zeta \\
& = -\frac{1}{\pi^2}  \frac{1}{w-z} \,\operatorname{p.v.}
\int_{\Gamma\setminus \Gamma_{z,\epsilon}} \frac{1}{\zeta-w}\,d\zeta
+ \frac{1}{\pi^2}  \frac{1}{w-z} \,\operatorname{p.v.}
\int_{\Gamma\setminus \Gamma_{z,\epsilon}} \frac{1}{\zeta-z} \,d\zeta \\
& = h(w)+k(w),
\end{align}
where in the last identity we defined $h(w)$ and $k(w)$.

Applying the bilipschitz character of  $\gamma$ we conclude that
\begin{equation}\label{ka}
|k(w)| \leq \frac{1}{\pi^2} \frac{L^2}{M \ep_0^2}
\operatorname{length}(\Gamma),\quad w  \in \Gamma\setminus
\Gamma_{z,M \epsilon}, \quad \ep_0 < \ep.
\end{equation}

The estimate of $h(w)$ is a little trickier. We have
\begin{equation}
h(w)= -\frac{1}{\pi^2}  \frac{1}{w-z} \,\operatorname{p.v.}
\int_{\Gamma} \frac{1}{\zeta-w}\,d\zeta + \frac{1}{\pi^2}
\frac{1}{w-z} \,\operatorname{p.v.} \int_{ \Gamma_{z,\epsilon}}
\frac{1}{\zeta-w}\,d\zeta
\end{equation}
\noindent A simple application of Cauchy's Theorem gives that, if
$\Gamma$ has a tangent at $w$,
\begin{equation}
\operatorname{p.v.} \int_{\Gamma} \frac{1}{\zeta-w}\,d\zeta = \pi i.
\end{equation}
As before, the bilipschitz character of $\gamma$ yields
\begin{equation}
|w-z|\geq \frac{M\ep}{L}, \quad w \in  \Gamma\setminus \Gamma_{z,M
\epsilon}
\end{equation}
and
$$
|w-\zeta| \geq |w-z|-|z-\zeta| \geq \ep \Big(\frac{M}{L}-L\Big), \quad w \in
\Gamma\setminus \Gamma_{z,M \epsilon} \quad \zeta \in
\Gamma_{z,\epsilon}
$$
Choose $M$ so that $M/L-L \geq 1.$ Then
$$
|w-\zeta| \geq \ep,  \quad w \in \Gamma\setminus \Gamma_{z,M
\epsilon} \quad \zeta \in \Gamma_{z,\epsilon}.
$$
Gathering all the previous estimates we finally get
\begin{equation}\label{hac}
|h(w)| \leq \frac{1}{\pi}\frac{L}{M \ep_0}+ \frac{1}{\pi^2}
\frac{\operatorname{length(\Gamma)}}{\ep_0}, \quad w \in
\Gamma\setminus \Gamma_{z,M \epsilon}, \quad \ep_0 < \ep.
\end{equation}
Hence \eqref{ka} and \eqref{hac} yield
\begin{equation}
|g_{z,\ep}(w)| \leq C, \quad w \in \Gamma\setminus \Gamma_{z,M
\epsilon}, \quad \ep_0 < \ep,
\end{equation}
where $C=C(\ep_0, M, L, \operatorname{length}(\Gamma))$ is a
constant depending on $\ep_0, M, L$ and
$\operatorname{length}(\Gamma)$.

Therefore
$$
\Big|\int_{\Gamma\setminus \Gamma_{z,M
\epsilon}}Tf(w)\,g_{z,\epsilon}(w)dw \Big| \leq C\, \int_{\Gamma}
|Tf(w)| \,d\mu(w) \leq C \operatorname{length}(\Gamma) M(Tf)(z),
$$
which completes the proof of the lemma.
\end{proof}

\section{The proof of the Theoem}
For $z \neq 0$ let $\operatorname{Arg}(z)$ denote the principal
argument of $z$, so that $ 0\leq \operatorname{Arg}(z) < 2 \pi.$
\begin{lemm}\label{ml}
Given $\alpha > 0$ there exists a positive number $\ep_0=\ep_0 (L)$
with the following property. Assume that $0 < \ep_1 \leq \ep_0$,
$\ep_1/2 < \ep \leq \ep_1$ and that for a fixed $x \in \mathbb{R}$ we
have $\gamma(x)=0$. If $\gamma(x-\tau), \; \tau > 0,$ satisfies
\begin{equation}\label{sec}
\frac{\ep_1}{2 L} < |\gamma(x-\tau)|< L \ep_1,
\end{equation}
then, for some $\theta$ such that
$\gamma(x-\tau)=|\gamma(x-\tau)| e^{i \theta}$, we have
\begin{equation}\label{arg}
\big|\theta-\big(\operatorname{Arg}(\gamma(x+\ep))+\pi\big)\big| < \alpha.
\end{equation}
\end{lemm}

\begin{proof}
Consider the triangle with vertices $0, \gamma(x-\tau)$ and
$\gamma(x+\ep)$ and side lengths $A= |\gamma(x-\tau)|$,
$B=|\gamma(x+\ep)|$ and $C= |\gamma(x+\ep)-\gamma(x-\tau)|$. By the
cosine Theorem
$$
C^2 = A^2+B^2 -2 AB \cos(\phi),
$$
where $\phi$ is the angle opposite to the side $C.$  In other
terms
$$
1+ \cos(\phi) = \frac{(A+B-C)(A+B+C)}{2AB}.
$$
By asymptotic quasiconformality, given $\delta>0$ there exists $\eta_0
> 0$ such that \\ $C = |\gamma(x+\ep)-\gamma(x-\tau)| <\eta_0$ implies $A+B \leq (1+\delta)C.$
The bilipschitz property of $\gamma$ \eqref{bilipschitz} yields
$\ep_1/2L^2 \leq \tau \leq L^2 \ep_1.$ Hence
\begin{equation}
1+ \cos(\phi)\leq \delta L^4 \frac{(\ep_1+\tau)^2}{\ep_1 \tau} \leq
2 \delta  L^6 (1+L^2)^2.
\end{equation}
Taking $\theta = \operatorname{Arg}(\gamma(x+\ep)) +\phi$ we see that $ |\theta-(\operatorname{Arg}(\gamma(x+\ep))+ \pi )| < \alpha $ 
provided $\delta$ is small enough. Since
$$
|\gamma(x+\ep)-\gamma(x-\tau)| \leq L (\ep + \tau) \leq \ep_0 L
(1+L^2),
$$
one has to choose $\ep_0$ so that $  \ep_0  L (1+L^2)\leq \eta_0,$
which shows the correct dependence of $\ep_0$ and completes the
proof of the Lemma.
\end{proof}
Given a point $z \in \Gamma$ we want now to construct a special Jordan arc $\Delta_z$ connecting $z$ to $\infty$ in the complement
of $\Gamma.$ Assume, without loss of generality, that $z=0$. Take $x \in \mathbb{R}$ with $\gamma(x)=0$. Let $\ep_0$ be the number
given in the preceding Lemma and define, for $j=0,1,2, \dots$, a polar rectangle by
\begin{equation}\label{polarR}
 R_j = \Big\{w = |w|e^{i \theta}: \frac{\ep_0}{2^{j+1} L} < |w| < \frac{\ep_0 L}{2^{j}} \quad \text{and}\quad
 \Big|\theta- {\rm{Arg}}\Big(\gamma\Big(x+\frac{\ep_0}{2^j}\Big)\Big)+\pi\Big| < \alpha\Big\}.
\end{equation}
\noindent Applying Lemma \ref{ml} with  $\ep=\ep_1=\ep_0/2^j$ we conclude that
$$
\{\gamma(x-\tau) : 0< \tau \} \cap \Big\{w : \frac{\ep_0}{2^{j+1} L} < |w| < \frac{\ep_0 L}{2^{j}} \Big\} \subset R_j.
$$
\noindent We need to introduce another polar rectangle
$$
S_j = R_j \cap \Big\{w : \frac{\ep_0 L}{2^{j+1}} < |w|\Big\}, \quad j=0,1,2, \dots
$$

We define inductively $\Delta_z =\Delta_0$ on $S_j$ by just
requiring that the Jordan arc $ \Delta_0 \cap \overline{S_j} $ lies
in the unbounded component of the complement of $\Gamma$,
$\overline{S_j} $ being the closure of $S_j.$ We then connect
$\Delta_0 \cap \overline{S_0} $ with $\infty$ by a Jordan arc in the
complement of $\Gamma,$ with the only precaution of not reentering the
disc $D(0,\ep_0)$ once $\Delta_0$ has left it.

It is worth pointing out that the axis of two consecutive polar rectangles $R_j$ and $R_{j+1}$ make an angle less than $\alpha$.
This follows by the defining property of $\ep_0$ (see the proof of Lemma \ref{ml}).

\begin{lemm}\label{argument}
\begin{equation}\label{log}
 \log(\gamma(x-\ep))-\pi i = \log(-\gamma(x-\ep)), \quad x \in \mathbb{R}, \quad 0 < \ep \leq \ep_0.
\end{equation}
\end{lemm}
\begin{proof}
We know that
 \begin{equation}\label{diflog}
  \log(\gamma(x-\ep))-\pi i = \log(-\gamma(x-\ep))+ 2 \pi m i
 \end{equation}
for some integer $m$. Our goal is to compute the difference
$$
\log(\gamma(x-\ep))-\log(-\gamma(x-\ep))
$$
by the integral
$$
 \int_{\varsigma} \frac{1}{z} \,dz,
$$
where $\varsigma$ is an appropriately chosen Jordan arc connecting
$-\gamma(x-\ep)$ to $\gamma(x-\ep)$ in the complement of $\Delta_0.$

\begin{figure}[h!]

\caption{The curve $\varsigma$}
\vspace{0.5cm}
\centering
\includegraphics[width=0.75\textwidth]{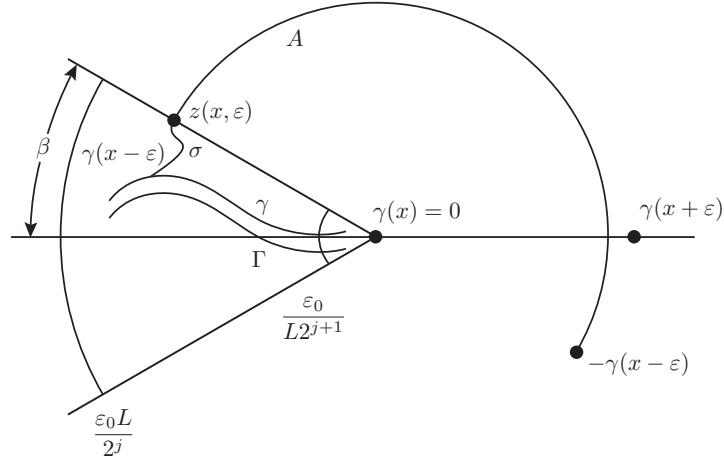}
\end{figure} 

Assume that $\ep_0 / 2^{j+1} < \ep \leq \ep_0 / 2^{j}$, for some
non-negative integer $j.$  Define $N$ as the smallest integer
satisfying
$$
\frac{L \ep_0}{2^{j+N}} \leq \frac{\ep_0}{L 2^{j+1}}.
$$
This is equivalent to  $L^2 \leq 2^{N-1}$ and so $N$ depends only on
$L.$ Hence $R_k \subset D(0,{\ep_0}/{L \,2^{j+1}}),\; k \geq
j+N,$ and, in particular, $R_k,\; k \geq j+N,$ does not intersect
the circumference  $\partial D(0,|\gamma(x-\ep)|).$

The angle between the axis of the polar rectangle $R_{j+l}$ and that
of $R_j$ is not greater than $l \alpha \leq N \alpha, \; l=1, 2,...,
N-1.$ Set $\beta = N \alpha,$ so that $\beta$ can be as small as
desired by taking $\alpha= \alpha(L)$ appropriately. We conclude
that
$$
R_{j+l}\subset \{w: w= |w| e^{i \theta} \;\, \text{with}\;\, |\theta
- \operatorname{Arg}(\gamma(x+\ep)+\pi)| < \beta \}, \quad l =
1,2,\dots,N-1.
$$

We are now ready to define the Jordan arc $\varsigma.$ Let
$z(x,\ep)$ be the point at the intersection of the circumference
$\partial D(0,|\gamma(x-\ep)|)$ and the ray
$$\{w: w= |w|e^{i \theta} \;\, \text{with}\;\, \theta= \operatorname{Arg}(\gamma(x+\ep)+\pi)- \beta \}.$$
\noindent Let $A$ stand for the arc in  $\partial
D(0,|\gamma(x-\ep)|)$ having $-\gamma(x-\ep)$ as initial point and
$z(x,\ep)$ as end point (counterclockwise oriented).

There exists a rectifiable Jordan arc $\sigma$ joining the points
$z(x,\ep)$ and $\gamma(x-\ep)$ in the bounded component of the
complement of $\Gamma$ with the property that
$$\operatorname{length}(\sigma) \leq C\, |z(x,\ep)-\gamma(x-\ep)|.$$
This can be seen readily as follows. Set $\tilde{\gamma}(e^{ix})=
\gamma(x),\; x \in \mathbb{R}.$  Then $\tilde{\gamma}$ is a
bilipschitz homeomorphism  between $\mathbb{T}$ and $\Gamma$ and
thus can be extended to a global bilipschitz homeomorphism of the
plane onto itself (see \cite{Tuk1},\cite{Tuk2}). The existence of the arc $\sigma$ is then
easily proved by transferring the question via the extended
bilipschitz homeomorphism.

Define $\varsigma = A \cup \sigma,$ oriented as already specified.
Note that $\varsigma$ lies in the complement of $\Delta_0,$ by the
previous discussion, in particular, the definition of $N$ and
$\beta.$
Therefore
\begin{equation}\label{varlog}
\log(\gamma(x-\ep))-\log(-\gamma(x-\ep)) =  \int_{\varsigma}
\frac{1}{z} \,dz.
\end{equation}
On one hand we have
\begin{equation}
\int_{A} \frac{1}{z} \,dz = \pi i + O(\beta)
\end{equation}
and on the other hand
\begin{equation}
\Big|\int_{\sigma} \frac{1}{z} \,dz\Big| \leq
\frac{C\,|z(x,\ep)-\gamma(x-\ep)|}{|\gamma(x-\ep)|} \leq C\,
\beta=O(\beta).
\end{equation}
If $\beta$ is small enough so that $O(\beta)< \pi,$ then, by
\eqref{diflog}, we get that $m=0,$ and the lemma is proved.
\end{proof}

We need a final lemma, which concludes the proof of the Theorem.
\begin{lemm}\label{efa}
Let $\Gamma$ be an  asymptotically conformal chord-arc curve and let $\gamma$ be a bilipschitz parametrization of $\Gamma$ (in the sense of \eqref{bilipschitz}).
Then there exists a constant $C > 1$ and a positive number $\ep_0$ such that
\begin{equation}\label{efasegon}
C^{-1}\, \frac{|\gamma(x+\ep)+\gamma(x-\ep)- 2 \gamma(x)|}{\ep} \leq |F(x,\ep)|
\leq C \, \frac{|\gamma(x+\ep)+\gamma(x-\ep)- 2 \gamma(x)|}{\ep},
\end{equation}
for $x \in \mathbb{R}$ and $0< \ep < \ep_0.$
\end{lemm}
\begin{proof}
Without loss of generality assume that $\gamma(x)=0.$  Let $\ep_0$ be the small number provided by Lemma \ref{ml}. By the construction of
the arc $\Delta_0$ described  in the proof of Lemma \ref{ml} we have that the segment joining $-\gamma(x-\ep)$ and $\gamma(x+\ep)$ lies in the complement of $\Delta_0.$  We have, by Lemma \ref{argument},
\begin{align}
F(x,\epsilon) & = \log\big(\gamma(x+\epsilon)\big) - \log\big(\gamma(x-\epsilon)\big) + \pi i \\ &
= \log\big(\gamma(x+\epsilon)\big) - \log\big(-\gamma(x-\epsilon)\big)
\end{align}
and so
\begin{align}
F(x,\epsilon) & = \int_0^1 \frac{d}{dt} \log\big(-\gamma(x-\epsilon)+ t(\gamma(x+\epsilon)+\gamma(x-\epsilon)) \big) \,dt\\ &
= \int_0^1 \frac{\gamma(x+\epsilon)+\gamma(x-\epsilon)}{-\gamma(x-\epsilon)+ t(\gamma(x+\epsilon)+\gamma(x-\epsilon))}  \,dt.
\end{align}
Set, to simplify notation, $a= -\gamma(x-\epsilon)$, $b=\gamma(x+\epsilon)$  and let $\theta$ denote the angle between $a$ and $b$. By Lemma \ref{ml} we know that $\theta$ is as small as we wish. In particular we can assume that $\cos(\theta) \geq 1/2.$ Thus, using the cosine Theorem,
\begin{align}
|a+t(b-a)|^2 & = (1-t)^2 |a|^2 + t^2 |b|^2 + 2 (1-t) t |a| |b| \cos(\theta) \\ & \geq \frac{1}{2}\,((1-t)|a|+t |b| )^2 \geq \frac{\ep^2}{2L^2},
\end{align}
and
\begin{equation}
|F(x,\ep)| \leq  \frac{\sqrt{2} L }{\ep} |\gamma(x+\ep)+\gamma(x-\ep)|,
\end{equation}
which is the upper estimate in \eqref{efasegon}.

For the lower estimate we set $z_t= -\gamma(x-\epsilon)+ t(\gamma(x+\epsilon)+\gamma(x-\epsilon)). $  Since
$ \operatorname{Re}(z_t) \geq  |z_t|/2 $ and $  |z_t|  \geq  \ep / (\sqrt{2} L)$
\begin{align}\label{efalower}
 \Big| \int_0^1 \frac{1}{z_t}  \,dt \Big|  \geq
\operatorname{Re} \int_0^1 \frac{1}{z_t}  \,dt  & =
 \int_0^1 \frac{\operatorname{Re} (z_t) }{| z_t|^2}  \,dt  \\&   \geq \int_0^1   \frac{1}{2| z_t|}\,dt  \geq
\frac{L}{\sqrt{2}  \ep}.
\end{align}
To complete the proof of the Theorem one only needs to combine
Lemmas \ref{Flog}, \ref{smalltrunc} and \ref{efa}.
\end{proof}

%
%
%
%
%
%
%

\begin{rem*}\label{rem_decay}
Let  $a= \gamma(x)-\gamma(x-\epsilon)$, $b=\gamma(x+\epsilon)-\gamma(x)$ and let $\alpha(x,\epsilon)$ be the angle spanned by $a$ and $b$.
By geometric considerations and using the cosine Theorem, 
one can see that for a bilipschitz parametrization $\gamma$ such that
\begin{equation}\label{bilipschitz2}
  c\,|x-y| \leq |\gamma(x) -\gamma(y)| \leq C \,|x-y|, \quad x, y \in \mathbb{R}, \quad |x-y| \leq \frac{T}{2},
\end{equation}
we get the a priori estimate
\begin{equation}
|\gamma(x+\ep)+\gamma(x-\ep)- 2 \gamma(x)|^{2}\leq c^2\epsilon^2+C^2\epsilon^2-2cC\epsilon^2 \cos\alpha(x,\epsilon).
\end{equation}
So, in the general case, we can guarantee just a linear decay of the second finite difference $|\gamma(x+\epsilon)+\gamma(x-\epsilon)-2\gamma(x)|$ and the logarithmic condition \eqref{snddiff} gives informations about the local behavior of the best constants $c$ and $C$ around $x$ and about the decay of $\alpha(x,\epsilon)$ for $\epsilon$ small. This remark will be useful in the next section.
\end{rem*}

\section{An example}

In this section we provide an example of curve $\gamma$ which is not $C^1$ but for which the improved Cotlar's inequality \eqref{m2} holds. The curve will be constructed in a recursive way and will be parametrized by arc-length. Without loss of generality, we will focus on defining a curve which is not closed. Indeed, possibly by connecting the ends of this curve in a smooth way, we can reduce to the same environment of the previous sections.\\
Let $0<\alpha<\pi/2.$ Let $F_{\alpha}:[0,1]\to	 \mathbb{R}$ be the function with support in $[1/4,3/4]$ which is linear in $[1/4,1/2]$ and $[1/2,3/4]$ with slope $\tan\alpha$ in $[1/4,1/2]$ and $-\tan\alpha$ in $[1/2,3/4]$. In other words
\begin{equation}
F_{\alpha}(t):=\max\Big\{0, \Big(\frac{1}{4}-\Big|t-\frac{1}{2}\Big|\Big)\tan\alpha\Big\}.
\end{equation}
Let $\xi>0$. For $t\in\mathbb{R}$ we define the function
\begin{equation}
\eta_\xi(t):=\eta\Big(\frac{t}{\xi}\Big)\frac{1}{\xi},
\end{equation}
where $\eta$ is a smooth, even and positive function such that $\supp\eta\subset [-1,1]$ and $\int\eta(t)dt=1.$
For $0<\xi<1/100$ we define the regularized function
\begin{equation}
\lambda_{\alpha}:=F_\alpha*\eta_\xi.
\end{equation}
We will call the curve $\Lambda_{\alpha}:=\big(t,\lambda_{\alpha}(t)\big)_{t\in[0,1]}$ $\alpha$-\textit{patch.}\\
An $\alpha-$patch has the following properties:
\begin{itemize}
\item $\Lambda_{\alpha}$ is the graph of a function $\lambda_{\alpha}:[0,1]\rightarrow \mathbb{R}$ which is symmetric around $1/2$.
\item if we denote by $[a,b]$ the segment joining the points $a,b\in\mathbb{R}^2,$ then $\Lambda_{\alpha}$ contains the segments $I_{\alpha}:=[(0,0\big),(1/4-\xi,0)],$ $II_{\alpha}:=[(1/4+\xi,\xi\tan\alpha),(1/2-\xi,(1/4-\xi)\tan\alpha)],$ $III_{\alpha}:=[(1/2+\xi, (1/4-\xi)\tan\alpha),(3/4-\xi,\xi\tan\alpha)]$ and $IV_{\alpha}:=[(3/4+\xi,0),(1,0)].$ We denote by $C^i_\alpha,$ $i=1,2,3$ the remaining three non-affine parts of the graph. Precisely, $C^1_\alpha$ joins the segments $I_\alpha$ and $II_\alpha,$ $C^2_\alpha$ the segments $II_\alpha$ and $III_\alpha$ and $C^3_\alpha$ the segments $III_\alpha$ and $IV_\alpha.$
\item the function $\lambda_\alpha$ is convex on the intervals below $C^1_\alpha$ and $C^3_\alpha$ and concave on the interval below $C^2_\alpha.$
\end{itemize}
The idea is that the $\alpha-$patch is a smoothened corner, as shown in Figure 2.
\begin{figure}[h]
\label{fig:non_c1}
\centering
\begin{tikzpicture}

\draw[thick]
  (0,0) coordinate (a) node[below left] {0}
  -- (1,0) coordinate (c) node[above] {$I_{\alpha}$}
  -- (1.7,0) coordinate (b) node[below right] {$\frac{1}{4}$};

\draw[thick]
  (3,1.47) coordinate (c) node [above left] {$II_{\alpha}$}
  (2.3,0.44) coordinate (a) node[above left] {}

  -- (3.7,2.5) coordinate (b) node[below right] {};

\draw[thick]
(5,1.47) coordinate (c) node [above right] {$III_{\alpha}$}
  (4.3,2.5) coordinate (a) node[below left] {}
  -- (5.7,0.44) coordinate (b) node[below right] {};
  
\draw[thick]
  (6.3,0) coordinate (a) node[below] {$\frac{3}{4}$}
  -- (7,0) coordinate (c) node [above] {$IV_{\alpha}$}
  -- (8,0) coordinate (b) node[below right] {1}; 
  
\draw [dashed]
  (0,0) coordinate (a) node[below left] {}
  -- (8,0) coordinate (b) node[below right] {};
  \draw [dashed]
  (3.7,2.5) coordinate (a) node[above right] {}
  -- (2,0) coordinate (b) node[above left] {}
  -- (5,0) coordinate (c) node[above right] {}
  pic["\color{blue}$\alpha$",draw=blue, thick ,-,angle eccentricity=1.7,angle radius=0.7cm] {angle=c--b--a};

   \draw [thick,red] (1.7,0) to [ curve through ={(1.9,0.07)}] (2.3,0.44);
   \draw [thick,red] (3.7,2.5) to [ curve through ={(4,2.7)}] (4.3,2.5);
      
   \draw [thick,red] (5.7,0.44) to [ curve through ={(6.1,0.07)}] (6.3,0);
\end{tikzpicture}
\caption{An $\alpha-$patch}
\end{figure}
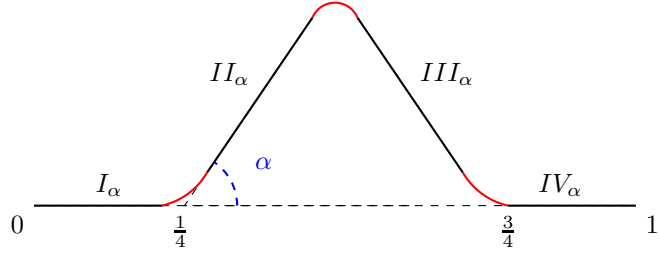
\begin{rem}\label{size_tau}
Let us denote by $\tau(\alpha)$ the difference between the length of the (non-smoothened) graph of $F_{\alpha}$ and the length of $\Lambda_\alpha$. For what follows, we need to estimate its behavior for small values of $\alpha$. It suffices to observe that
\begin{equation}\label{tau_error}
\begin{split}
\tau(\alpha)&:=\length(F_\alpha)-\length(\Lambda_\alpha)\\
&=\int^1_0\Big(\sqrt{1+|f'_\alpha*\eta_\xi|^2(t)}\Big)-\Big(\sqrt{1+|f'_\alpha|^2(t)}\Big)dt\\
&=\int^1_0\frac{|f'_\alpha*\eta_\xi|^2(t)-|f'_\alpha|^2(t)}{\Big(\sqrt{1+|f'_\alpha*\eta_\xi|^2(t)}\Big)+\Big(\sqrt{1+|f'_\alpha|^2(t)}\Big)}dt\leq 2||f'_\alpha||^2_\infty = 2\tan\alpha.
\end{split}
\end{equation}
\end{rem}
\subsection*{Definition of the curve $\Gamma$}
Let $\alpha_j:=1/j$ for $j=1,2,\ldots$ positive integer. For the sake of notational convenience we replace the subscript $\alpha_j$ by $j$; for instance, we write $\Lambda_j$ for $\Lambda_{\alpha_j}$, $I_j$ for $I_{\alpha_j},\ldots,IV_j$ for $IV_{\alpha_j}$ and $C^i_j$ for $C^i_{\alpha_j}$. Moreover, $\tau_j:=\tau(\alpha_j).$
Now we can define $\Gamma$ according to the following recursive steps:
\begin{itemize}
\item $\Gamma_1:=\Lambda_1.$
\item We would like to glue on $II_1$ an appropriate rescaled, translated and rotated copy $\tilde{\Lambda}_2$ of $\Lambda_2.$ The angle of rotation is $\alpha_1.$ The scaling factor and the translation are chosen so that the origin of $\tilde{\Lambda}_2$ is $(1/4,0)$ and the end is $\big(1/2, (\tan\alpha)/4\big).$
Denote by $\widetilde{II}_2$ the image of $II_2$ via the same affinity which maps $\Lambda_2$ to $\tilde{\Lambda}_2$; let us use the tilde to denote the images of the other parts of the patch via the same map, too. Delete the segment $II_1$ from $\Lambda_1$ and add $\tilde{\Lambda}_2.$ Now the endings of $\tilde{\Lambda}_2$ should be deleted in order to make a connection with $\Lambda_1.$ The precise expression for the second step curve is
\begin{equation}
\Gamma_2:=\big((\Lambda_1\setminus II_1)\cup \tilde{\Lambda}_2\big)\setminus \big((\tilde{I}_2\cup \widetilde{IV}_2)\setminus II_1\big).
\end{equation}
\item given $\Gamma_n,$ which is a ``gluing'' of affine copies $\tilde{\Lambda}_j$ of $\Lambda_j$ for $j\in\{1,\ldots,n\},$  where $\widetilde{II}_n$ is the image of $II_j$ under the same affinity which maps $\Lambda_j$ to $\tilde{\Lambda}_j$, we define
\begin{equation}
\Gamma_{n+1}:=((\tilde{\Lambda}_n\setminus \widetilde{II}_n)\cup \tilde{\Lambda}_{n+1})\setminus ((\widetilde{I}_{n+1}\cup \widetilde{IV}_{n+1})\setminus \widetilde{II}_n),
\end{equation}
where $\tilde{\Lambda}_{n+1}$ is an re-scaled copy of $\Lambda_{n+1}$ rotated by an angle $\sum_{j=1}^{n+1}\alpha_j$ whose vertices coincide with the images of $(1/4,0)$ and $\big(1/2, \tan(\alpha)/4\big)$ via the transformation of the plain that sends $\Lambda_n$ to $\widetilde{\Lambda}_n$.
\end{itemize}
Then, $\{\Gamma_n\}_n$ converges in the Hausdorff distance (a similar case is presented, for example, in \cite{Falconer}) and we can simply define $\Gamma:=\lim_n\Gamma_n.$
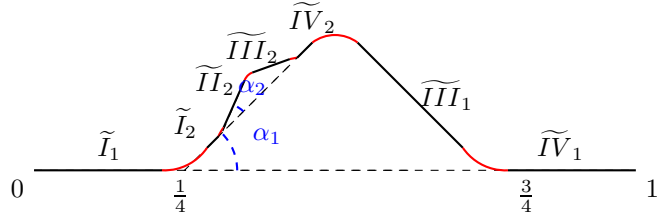
\begin{figure}[h]
\label{fig:exam}
\centering
\begin{tikzpicture}

\draw[thick]
  (0,0) coordinate (a) node[below left] {0}
  -- (1,0) coordinate (c) node[above] {$\widetilde{I}_1$}
  -- (1.7,0) coordinate (b) node[below right] {$\frac{1}{4}$};

\draw[thick]
  (2.3,0.3) coordinate (a) node[above left] {$\widetilde{I}_2$}

  -- (2.45,0.45) coordinate (b) node[below right] {};

  
\draw[thick]
  (2.52,0.57) coordinate (a) node[left] {}

  -- (2.8,1.2) coordinate (b) node[left] {$\widetilde{II}_2$};
  
 \draw[thick]
  (2.9,1.3) coordinate (a) node[above] {$\widetilde{III}_2$}

  -- (3.4,1.48) coordinate (b) node[above] {}; 

\draw [thick,red] (2.45,0.45) to [ curve through ={(2.48,0.51)}] (2.52,0.57);
\draw [thick,red] (2.8,1.2) to [ curve through ={(2.85,1.27)}] (2.9,1.3);
\draw [thick,red] (3.4,1.48) to [ curve through ={(3.45,1.49)}] (3.5,1.5);  
  
  
\draw[thick]
  (3.5,1.5) coordinate (a) node[above] {}

  -- (3.7,1.7) coordinate (b) node[above] {$\widetilde{IV}_2$};

\draw[thick]
  (4.3,1.7) coordinate (a) node[below right] {}
  -- (5,1) coordinate (c) node[right] {$\widetilde{III}_1$}
  -- (5.7,0.3) coordinate (b) node[above] {};
  
\draw[thick]
  (6.3,0) coordinate (a) node[below right] {$\frac{3}{4}$}
  -- (7,0) coordinate (c) node [above] {$\widetilde{IV}_1$}
  -- (8,0) coordinate (b) node[below right] {1}; 
  
\draw [dashed]
  (0,0) coordinate (a) node[below left] {}
  -- (8,0) coordinate (b) node[below right] {};
\draw [dashed]
  (3.7,1.7) coordinate (a) node[above right] {}
  -- (2,0) coordinate (b) node[above left] {}
  -- (5,0) coordinate (c) node[above right] {}
  pic["\color{blue}$\alpha_1$",draw=blue, thick ,-,angle eccentricity=1.7,angle radius=0.7cm] {angle=c--b--a};

\draw [dashed]
  (2.8,1.2) coordinate (a) node[above right] {}
  -- (2.5,0.5) coordinate (b) node[above left] {}
  -- (3.4,1.4) coordinate (c) node[above right] {}
  pic["\color{blue}$\alpha_2$",draw=blue, thick ,-,angle eccentricity=1.8,angle radius=0.4cm] {angle=c--b--a};

   \draw [thick,red] (1.7,0) to [ curve through ={(1.9,0.02)}] (2.3,0.3);
   \draw [thick,red] (3.7,1.7) to [ curve through ={(4,1.8)}] (4.3,1.7);
      
   \draw [thick,red] (5.7,0.3) to [ curve through ={(6.1,0.02)}] (6.3,0);
\end{tikzpicture}
\caption{The second step in the construction of the curve $\Gamma$}
\end{figure}
Let us now state an estimate that we will use in what follows.
\begin{lemm}\label{estim_angle}
Given $0<\alpha<\pi/2$ and $z_1,z_2\in\Lambda_{\alpha}$, we have
\begin{equation}\label{angle_corner}
l(z_1,z_2)\leq \frac{|z_1-z_2|}{\cos\alpha},
\end{equation}
where $l(z_1,z_2)$ denotes the length of the arc of $\Lambda_\alpha$ joining $z_1$ and $z_2$.
\end{lemm}
\begin{proof}
Let $t_1:=\lambda^{-1}_{\alpha}(z_1)$ and $t_2:=\lambda^{-1}_{\alpha}(z_2)$. We have $|t_1-t_2|\leq |z_1-z_2|.$ Moreover, because of the way we constructed $\Lambda_{\alpha},$ we have that $|\lambda'_{\alpha}(t)|\leq \tan \alpha$ for every $t\in [0,1].$ Collecting all these observations,
\begin{align}
l(z_1,z_2)&=\int_{t_1}^{t_2}\sqrt{1+|\lambda_{\alpha}'(t)|^2}dt\leq \int_{t_1}^{t_2}\sqrt{1+|\tan\alpha|^2}dt\\
&=|t_2-t_1|\sqrt{1+|\tan\alpha|^2}=\frac{|t_2-t_1|}{\cos\alpha}\leq \frac{|z_2-z_1|}{\cos\alpha}.
\qedhere
\end{align}

\end{proof}
\begin{rem}
Notice that the inequality \eqref{angle_corner} keeps holding for a scaling of $\Lambda_\alpha,$ in particular for the $\tilde{\Lambda}_j,\,j\in\mathbb{N}.$
\end{rem}
Let us define
\begin{equation}
L_n:= 2^{-2n+1}\Big(\prod_{j=1}^{n-1}\cos\alpha_j\Big)^{-1},
\end{equation}
which is half of the diameter of the rescaled patch $\tilde{\Lambda}_n$ in the construction of the curve $\Gamma.$ We will use $L_n$ as a quantifier of the scale.
\begin{lemm} \label{errors}For every $\delta>0$ there exists $k\in\mathbb{N}$ big enough such that for  $z_1,z_2\in \Gamma\cap (\bigcup_{j=k}^{\infty}\tilde{\Lambda}_j)$ we have
\begin{equation}\label{sm_length}
l(z_1,z_2)\leq (1+\delta)|z_1-z_2|.
\end{equation}
\end{lemm}
\begin{proof} Let us start with some geometrical observation. \\
Let $k\in\mathbb{N}$ and $\zeta_1,\zeta_2\in\Gamma$. Suppose, moreover, that $\zeta_1\in \tilde{I}_k$ and $\zeta_2\in\widetilde{IV}_k$. It is useful to define
\begin{equation}
R_k:=l(\zeta_1,\zeta_2)-|\zeta_1-\zeta_2|.
\end{equation}
Observe that the definition of $R_k$ does not depend on the choice of $\zeta_1$ and $\zeta_2$ in the respective segments. In particular, by the construction of the curve $\Gamma$ and by the definition of the error term $\tau_j$ in \eqref{tau_error}, it is not difficult to check that we have
\begin{equation}\label{err}
R_k= \Big(3 \sum_{j=k+1}^{\infty} L_j- L_k\Big)- \sum_{j=k+1}^{\infty} 2L_j\tau_j.
\end{equation}
The term between parentheses in the right hand side is the length of the gluing of the 'non-regularized' $\alpha-$patches in the construction and the second sum is an error term due to the smoothing in the definition of $\alpha$-patch.\\
Because of the how we chose $L_j$ and $\tau_j,$ the quantity $R_k$ represents the error we make in estimating the length of the arch of the curve between $\zeta_1\in \tilde{I}_k$ and $\zeta_2\in\widetilde{IV}_k$ compared to $|\zeta_1-\zeta_2|.$ The presence of factor $2L_j$ in the last sum in the right hand side of \eqref{err} is due to the fact that the diameter of $\tilde{\Lambda}_j$ is equal to $2L_j$ and, thus, the error term $\tau_j$ has to be rescaled by that value.
It turns out that
\begin{equation}\label{small_error}
\frac{R_k}{L_k}\rightarrow 0 \,\,\text{ as }k\rightarrow\infty.
\end{equation}
Indeed, we have
\begin{align}
\frac{3}{L_k} \sum_{j=k+1}^{\infty} L_j&=\sum_{j=k+1}^{\infty}\frac{3}{4^{j-k+1}}\Big(\prod_{l=k}^{j}\cos\alpha_j\Big)^{-1}\\ \nonumber
&\leq
\sum_{j=k+1}^{\infty}\frac{3}{4^{j-k+1}}(\cos\alpha_k)^{-(j-k)}=\frac{12 \cos\alpha_k}{4\cos \alpha_k-1}-3 
\end{align}
and the last term tends to $1$ as $k\rightarrow\infty.$ Moreover, using \eqref{tau_error} and since $L_j\leq 2^{k-j}L_k$ for $j>k$, we have that
\begin{equation}
\frac{1}{L_k}\sum_{j=k+1}^{\infty}2L_j\tau_j\lesssim \tau_{k+1} \sum_{j=k+1}^\infty {2^{k-j}}\rightarrow 0\,\,\text{ as }k\rightarrow\infty,
\end{equation}
so that \eqref{small_error} follows.\\
Let us combine this observation with \eqref{angle_corner} to prove \eqref{sm_length}. Let $z_1,z_2\in\Gamma$. Observe that each point of $\Gamma$ belongs to $\widetilde{\Lambda}_j$ for at most two different $j$.
Let $k_1$ be the maximum index such that $z_1\in\tilde{\Lambda}_{k_1}$ and let $k_2$ be the maximum index such that $z_2\in\tilde{\Lambda}_{k_2}$. The rest of the proof works with minor changes if we take the minimum instead of the maximum in the definitions of $k_1$ and $k_2.$\\
Without loss of generality, suppose $k_1\leq k_2.$
If $k_1=k_2$, the definition of $R_k$ and the estimate \eqref{angle_corner} allow us to write
\begin{equation}\label{easy_case}
l(z_1,z_2)\leq \frac{|z_1-z_2|}{\cos\alpha_{k_1}},
\end{equation}
if the point are at a distance $|z_1-z_2|\leq L_{k_1+1}$. For $|z_1-z_2|\geq L_{k_1+1}/4,$ we have to consider the additional error term $R_{k_1+1}$. In particular
\begin{equation}
l(z_1,z_2)\leq \frac{|z_1-z_2|}{\cos\alpha_{k_1}}+R_{k_1+1}\leq \frac{|z_1-z_2|}{\cos\alpha_{k_1}}+\frac{R_{k_1+1}}{4L_{k_1+1}}|z_1-z_2|,
\end{equation}
so that, invoking \eqref{small_error}, the lemma is proven in the case $k_1=k_2$.\\
Let us consider the other case, $k_1<k_2$. If $z_2\in\widetilde{\Lambda}_{k_1}$, \eqref{easy_case} easily applies. So we can suppose $z_2\not\in\widetilde{\Lambda}_{k_1}.$
In this case 
\begin{equation}\label{ineqL}
|z_1-z_2|\geq \frac{L_{k_1+1}}{4}.
\end{equation}
Let $z'_2\in\widetilde{II}_{k_1}$ be the orthogonal projection of $z_2$ on the segment $\widetilde{II}_{k_1}.$ Using the triangular inequality and denoting by
\begin{equation}\label{height}
h_{k_1+1}:=\min\{h: \widetilde{\Lambda}_{k_1+1}\subset [0,h]n_V+V \text{ for some affine line }V \text{ with normal }n_V\}
\end{equation}
the width of $\tilde{\Lambda}_{k_1+1},$ we have
\begin{equation}\label{eps_1}
|z_1-z'_2|\leq |z_1-z_2|+ h_{k_1+1}.
\end{equation}
Let us remark that, by construction of $\Gamma,$
\begin{equation}\label{ineq}
\frac{h_{k_1+1}}{L_{k_1+1}}\rightarrow 0 \text{ as } k\rightarrow\infty.
\end{equation}
Given $m\in\mathbb{N}$ and $u,v\in\Gamma_m$, it is useful to denote by $l_m(u,v)$ the length of the arc of $\Gamma_m$ joining $u$ and $v$.
Now we want to prove that
\begin{equation}\label{tbp}
l(z_1,z_2)\leq l_{k_1}(z_1,z'_2)+ R_{k_1+1}.
\end{equation}
Let us just consider the case $z_1\in \tilde{I}_{k_1},$ since the other cases are analogous.
If $z_{k_2}\in \tilde{I}_{k_1+1}$ or $z_{k_2}\in \widetilde{IV}_{k_1+1},$ \eqref{tbp} holds trivially because $z_2=z'_2.$ Otherwise, let $\zeta$ be a point on $\widetilde{IV}_{k_1+1}$ and let us consider the quantities $l(z_2,\zeta)$ and $|z'_2-\zeta|.$ Observe that the consideration below does not depend on the auxiliary point $\zeta$ of $\widetilde{IV}_{k_1+1}$ we choose. Clearly $l(z_2,\zeta)\geq |z'_2-\zeta|$ and, because of the definition of $R_{k_1+1},$ the equality
\begin{equation}
l(z_1,z_2)+l(z_2,\zeta)=R_{k_1+1}+l_{k_1}(z_1,z'_2)+|z'_2-\zeta|,
\end{equation}
holds. So
\begin{equation}
l(z_1,z_2)=l_{k_1}(z_1,z'_2)+R_{k_1+1}+(|z'_2-\zeta|-l(z_2,\zeta))\leq l_{k_1}(z_1,z'_2)+R_{k_1+1}.
\end{equation}
The proof of the lemma is now over: indeed using \eqref{angle_corner}, \eqref{small_error}, \eqref{ineqL}, \eqref{eps_1} and \eqref{ineq} we get
\begin{align}\label{toinvert}
\frac{l(z_1,z_2)}{|z_1-z_2|}&\leq \frac{l_{k_1}(z_1,z'_2)}{|z_1-z_2|}+\frac{R_{k_1+1}}{|z_1-z_2|} \leq \frac{|z_1-z'_2|}{|z_1-z_2|\cos\alpha_{k_1}}+\frac{R_{k_1+1}}{|z_1-z_2|}\\
&\leq \frac{1}{\cos\alpha_{k_1}} + \frac{4h_{k_1+1}}{\cos\alpha_{k_1}L_{k_1+1}}+\frac{4R_{k_1+1}}{L_{k_1+1}}\rightarrow 1 \,\,\text{ as } k_1\rightarrow \infty. \qedhere
\end{align}
\end{proof}
\begin{prop}\label{nonsmooth}
$\Gamma$ is asymptotically smooth but not $C^1.$
\end{prop}
\begin{proof}
Let $\tilde{z}'_j\in\Gamma$ be the image of the point $z_{\alpha_j}'$ via the map which sends $\Lambda_j$ to $\widetilde{\Lambda}_j$.
We have that the curve $\Gamma$ is not $C^1$ at the point $z_0:=\lim_j z_j,$ where $z_j$ is an arbitrary point of $\tilde{\Lambda}_j$. Indeed, by our choice of the angles in the construction, $\sum_j \alpha_j = +\infty$  and the curve spirals close to $z_0$.\\
Let us now turn prove that the curve is asymptotically smooth.\\
Notice that we may write $\Gamma=\Gamma_1\cup\Gamma_2\cup\{z_0\},$ where $\Gamma_1$ and $\Gamma_2$ are smooth curves. Then, for every couple of points $\{z_1,z_2\}$ in one of those two smooth components we can exploit the smoothness to state that for every $\delta$ there exists $\bar{\epsilon}$ such that for $\epsilon<\bar{\epsilon}$ and $|z_1-z_2|=\epsilon$ we have
\begin{equation}
l(z_1,z_2)\leq (1+\delta)\epsilon.
\end{equation}
This, together with the result of Lemma \ref{errors} concludes the proof.
\end{proof}

Let us consider the arc-length parametrization $\gamma$ of $\Gamma$. Being $\Gamma$ asymptotically smooth, $\gamma$ is bilipschitz. In particular,
\begin{equation}
\frac{1}{C}|x-y|\leq|\gamma(x)-\gamma(y)|\leq |x-y|
\end{equation}
for a constant $C>1$ and $x,y\in [0,L(\Gamma)].$
As in Remark \ref{rem_decay} we denote by $\alpha(x,\epsilon)$ the angle between the vectors $\gamma(x)-\gamma(x-\epsilon)$ and $\gamma(x+\epsilon)-\gamma(x)$.
Because of the geometrical considerations in Remark \ref{rem_decay}, we have that
\begin{equation}\label{geom}
|\gamma(x+\ep)+\gamma(x-\ep)- 2 \gamma(x)|^2\leq \epsilon^2\Big( \frac{1}{C^2}+1-\frac{2}{C}\cos\alpha(x,\epsilon)\Big)
\end{equation}
for $\epsilon>0$ and $x\in[0,L(\Gamma)].$ Now we want to prove the estimate
\begin{equation}
|\gamma(x+\ep)+\gamma(x-\ep)- 2 \gamma(x)|\lesssim \frac{\epsilon}{|\log\epsilon|}.
\end{equation}
Being $\Gamma$ smooth off the point $z_0$ and arguing as in \cite{G}, the logarithmic condition \eqref{logomega} and the estimate \eqref{m2} are satisfied off that point.
Hence it suffices to prove \eqref{m2} for $\gamma(x)\in\cup_{k\geq k_0}\tilde{\Lambda}_k\cap\Gamma$ and $k_0$ big enough.
To do that, we will study the behavior of the angle $\alpha(x,\epsilon)$ and of the local value of the bilipschitz constant of $\gamma$ close to the point $z_0.$
\\
Being the curve asymptotically smooth, as a corollary of Lemma \ref{ml} we know that $\alpha(x,\epsilon)\to 0$ for $\epsilon$ small. Then, the second factor in the right hand side of \eqref{geom} behaves as
\begin{equation}
1+\frac{1}{C^2}-\frac{2}{C}\cos\alpha(x,\epsilon)= \Big[1+\frac{1}{C^2}-\frac{2}{C}\Big]+\frac{1}{C}\alpha(x,\epsilon)^2 + O(\alpha(x,\epsilon)^4)
\end{equation}
for $\epsilon\to 0.$\\
Let $x_0:=\gamma^{-1}(z_0).$ For $\epsilon>0,$ we denote by $C_{\epsilon}$ the smallest constant such that
\begin{equation}
\frac{1}{C_{\epsilon}}|x-y|\leq|\gamma(x)-\gamma(y)|\leq |x-y|
\end{equation}
holds for $x,y\in [x_0-\epsilon,x_0+\epsilon],$ i.e. the local value of the lower bilipschitz constant close to $x_0$.\\
Using this notation, to our purposes it suffices to prove that
\begin{equation}\label{estim_angle}
|\alpha(x,\epsilon)|\lesssim \big|\log\epsilon\big|^{-1}
\end{equation}
and
\begin{equation}\label{coef_old}
\Big[1-\frac{1}{C_{\epsilon}}\Big]\lesssim \big|\log\epsilon\big|^{-1}
\end{equation}
for $\epsilon$ small and $\gamma(x)$ close enough to $z_0.$\\
In particular, instead of \eqref{coef_old} we will prove that the stronger estimate
\begin{equation}\label{estim_lip}
\Big[1-\frac{1}{C_{\epsilon}}\Big]\lesssim \big|\log\epsilon\big|^{-2}
\end{equation}
holds for $\epsilon$ small and $\gamma(x)$ close enough to $z_0.$
The following two lemmas respectively prove the estimate for the angle and the estimate for $C_\epsilon.$
\begin{lemm}\label{lemm_ang}
For every $\epsilon_0$ there exists an integer $k_0$ such that
\begin{equation}
|\alpha(x,\epsilon)|\lesssim |\log\epsilon|^{-1}
\end{equation}
for $\epsilon<\epsilon_0,$ $|x-x_0|<\epsilon_0$ and $\gamma(x-\epsilon)\in\bigcup_{k=k_0}^{\infty}\widetilde{\Lambda}_k\cap\Gamma.$
\end{lemm}
\begin{proof}
Let $\epsilon>0$ and $z=\gamma(x)\in \Gamma.$ Moreover, let us define $z_{\pm}:=\gamma(x\pm \epsilon).$
Let $k$ be the maximum index such that $z\in\tilde{\Lambda}_{k}$ and let $k_\pm$ be the maximum index such that $z_\pm\in\tilde{\Lambda}_{k_\pm}$. Without loss of generality, we will prove the lemma for $x<x_0.$ Let us proceed with some geometrical consideration.\\
\begin{figure}[h!]
\caption{A schematic representation of the setting of the proof.}
\vspace{0.5cm}
\centering
\includegraphics[width=0.65\textwidth]{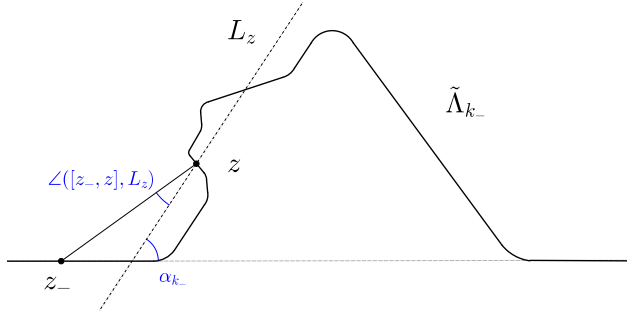}
\end{figure} 
Let $L_z$ denote the line passing through $z$ and parallel to the segment $\widetilde{II}_{k_-}.$
Due to the definition of the angle $\alpha(x,\epsilon)$, we can fix the line $L_z$ and bound $|\alpha(x,\epsilon)|$ by the absolute value of the smallest angle $\angle([z_-,z],L_z)$ that $L_z$ forms with the segment $[z_-,z]$ plus the absolute value of the smallest angle $\angle([z,z_+],L_z)$ that $L_z$ forms with the segment $[z,z_+].$\\
If $z$ belongs to $\widetilde{\Lambda}_{k_-},$ due to the properties of the $\alpha_{k_-}-$patch, the arc $\gamma([x-\epsilon, x])$ is entirely contained in a cone of vertex $z$ and aperture $\angle([z_-,z],L_z)$. By elementary geometric considerations, we can write
\begin{equation}\label{angle_1}
\big|\angle([z_-,z],L_z)\big|\leq \alpha_{k_-}.
\end{equation}
Again, due to few geometric observations (that are not substantial for the sequel and we decide to omit in order to make the proof more concise) and to the way $\Gamma$ is defined, it is not difficult to see that
\begin{equation}\label{angle_2}
\big|\angle([z_+,z],L_z)\big|\leq 2\alpha_{k_-}.
\end{equation}
We are left to consider the case $z\not\in\widetilde{\Lambda}_{k_-}.$ As we observed in Lemma \ref{errors}, in this case we have $|z_--z|\geq L_{k_-+1}/4.$ Moreover, $\bigcup_{j=k_-+1}^{\infty}\widetilde{\Lambda}_j\cap \Gamma$ is contained in a rectangle whose base lays on $\widetilde{II}_{k_-},$ whose length is smaller than, say, $5L_{k_-+1}/3$ and with height $h_{k_-+1}$ (for its definition we refer to \eqref{height} in Lemma \ref{errors}). We recall that
\begin{equation}
\frac{h_j}{L_j}\to 0 \text{ for } j\to\infty.
\end{equation}
Now observe that $z_+\in\bigcup_{j=k_-}^{\infty}\widetilde{\Lambda}_j\cap \Gamma.$ For every point $z$ in this rectangle, using that $|z-z_+|\gtrsim L_{k_-+1}$, it holds that
\begin{equation}\label{angle_3}
|\angle([z_-,z],L_z)|\lesssim \alpha_{k_-}
\end{equation}
and
\begin{equation}\label{angle_4}
|\angle([z,z_+],L_z)|\lesssim \alpha_{k_-}.
\end{equation}
Joining \eqref{angle_1},\eqref{angle_2},\eqref{angle_3} and \eqref{angle_4}, we get
\begin{equation}
|\alpha(z,\epsilon)|\lesssim \alpha_{k_-}.
\end{equation}
Then, by the construction of $\Gamma$ and the definition of $L_m$, $L_{m+1}/L_m\leq 1/2$ for every $m$, that by iteration leads to
\begin{equation}
L_m\leq 2^{-m}.
\end{equation}
Now, if $\gamma(x-\epsilon)\in \widetilde{\Lambda}_{k_-}$ for $k_-$ big enough, we have that $\epsilon\lesssim L_{k_-}$ so that
\begin{equation}
k_-\gtrsim  |\log\epsilon|
\end{equation}
for $\epsilon$ small enough.
So, gathering all the considerations and recalling that $\alpha_{k_-}=1/k_-$, we get the desired result.
\end{proof}

\begin{lemm}There exists $\epsilon_1>0$ such that the inequality \eqref{estim_lip} holds for $\epsilon<\epsilon_1$.
\end{lemm}
\begin{proof}
Let us consider $z_1,z_2\in\Gamma$. Let $k_1$ be the maximum index such that $z_1\in\tilde{\Lambda}_{k_1}$ and $k_2$ the maximum index such that $z_2\in\tilde{\Lambda}_{k_2}$. Without loss of generality, $k_1\leq k_2$ and $\gamma^{-1}(z_1)\leq \gamma^{-1}(\tilde{z}).$ The idea is to prove that $C^{-1}_{\epsilon}$ is greater than a quantity which approximates $\cos\alpha_{k_1}$. It is convenient to split the study into different cases.\\
If $k_1=k_2$ and $\gamma^{-1}(z_2)<\bar{x}$ or $k_2=k_1+1$ and $z_2\in\widetilde{I}_{k_1+1},$ then \eqref{angle_corner} gives
\begin{equation}
|z_1-z_2|\geq \cos\alpha_{k_1}l(z_1,z_2).
\end{equation}
If $k_1=k_2$ and $\gamma^{-1}(z_2)>\bar{x}$ or $k_2=k_1+1$ and $z_2\in\widetilde{IV}_{k_1+1}$, then we can write
\begin{equation}
|z_1-z_2|\geq \cos\alpha_{k_1}\big(l(z_1,z_2)-R_{k_1+1}\big)=\Big(\cos\alpha_{k_1}-\cos\alpha_{k_1}\frac{R_{k_1+1}}{l(z_1,z_2)}\Big)l(z_1,z_2)
\end{equation}
and we recall that
\begin{equation}
\frac{R_{k_1+1}}{l(z_1,z_2)}\lesssim \frac{R_{k_1+1}}{L_{k_1+1}}\to 0 \qquad\text{ for }\qquad k_1\to\infty.
\end{equation}
In the remaining cases, we know from the proof of Lemma \ref{errors} that
\begin{equation}
|z_1-z_2|\geq \Big(\cos\alpha_{k_1}-\cos\alpha_{k_1}\frac{h_{k_1+1}}{l(z_1,z_2)}-\cos\alpha_{k_1}\frac{R_{k_1+1}}{l(z_1,z_2)}\Big)l(z_1,z_2),
\end{equation}
so that, using the same argument as at the end of the proof of Lemma \ref{lemm_ang} together with the Taylor expansion for the cosine, the proof is completed. Let us remark that the exponent $2$ in \eqref{estim_lip} appears because of the Taylor expansion.
\end{proof}
The two previous lemmas show that the arc-length parametrization $\gamma$ of $\Gamma$ is such that the estimate
\begin{equation}
T_*(f)(z)\lesssim M^2(Tf)(z)
\end{equation}
holds for every $z\in \Gamma.$
\subsection*{Final remarks on the curve $\Gamma$.}
The curve $\Gamma$ that we studied in this section can be considered as an example of a critical curve for which the main theorem holds. Indeed, another look at the estimates we got tells that most of those concerning the geometry of the curve are close to being sharp. Moreover, the finite second difference $|\gamma(x+\epsilon)+\gamma(x-\epsilon)-2\gamma(x)|$ has the right decay we need; the choice of a slower decay for the angles $\alpha_j$ causes worse estimates for $|\alpha(x,\epsilon)|$ and, hence, the finite second difference estimate to fail. Let us notice that the spiraling of $\Gamma$ close to the point $z_0$ also gives an idea of how the critical curves may look like.\\
Asymptotically smooth curves that are not $C^1$ may also be defined by means of complex analysis (exploiting, for example, the results in \cite{Pomm}) but we found a constructive approach more convenient to our purposes.

\label{Bibliography}

\end{document}